\documentclass[12 pt]{amsart}

\usepackage[english]{babel}
\usepackage{amssymb}
\usepackage{amsfonts}
\usepackage{amsmath}
\usepackage{amsthm}
\usepackage{epsfig}
\usepackage{color}
\usepackage{amscd}
\usepackage[all]{xy}
\usepackage{hyperref}
\usepackage{float}

\newtheorem{lemma}{Lemma}[section]

\theoremstyle{definition}
\newtheorem{defn}[lemma]{Definition}

\theoremstyle{remark}

\newtheoremstyle{letter}{}{}{\itshape}{}{\bfseries}{.}{.5em}{#1\thmnote{ #3}}
\theoremstyle{letter}
\newtheorem*{letteredthm}{Theorem}

\usepackage[letterpaper,left=1in,right=1in,top=1in,bottom=1in]{geometry}
\author{\tiny{Alexander Kolpakov}}
\address{\tiny{Institut de math\'ematiques, Rue Emile-Argand 11, 2000 Neuch\^atel, Switzerland}}
\email{\tiny{kolpakov (dot) alexander (at) gmail (dot) com}}

\author{\tiny{Alexey Talambutsa}}
\address{\tiny{Steklov Mathematical Institute of RAS, 8 Gubkina St., 119991 Moscow, Russia}}
\address{\tiny{National Research University Higher School of Economics, 3 Kochnovsky Proezd, 125319 Moscow, Russia}}
\email{\tiny{alexey (dot) talambutsa (at) gmail (dot) com}}

\title[Spherical and geodesic growth rates of RACG's and RAAG's are Perron numbers]{Spherical and geodesic growth rates of right-angled Coxeter and Artin groups are Perron numbers}

\begin{document}

\vspace*{-1\baselineskip}

\begin{abstract}
We prove that for any infinite right-angled Coxeter or Artin group, its spherical and geodesic growth rates (with respect to the standard generating set) either take values in the set of Perron numbers, or equal $1$. Also, we compute the average number of geodesics representing an element of given word-length in such groups.\\

\noindent
\textit{Key words: } Coxeter group, Artin group, graphs of groups, RACG, RAAG, Perron number, algebraic integer, finite automaton.\\

\noindent
\textit{2010 AMS Classification: } 57R90, 57M50, 20F55, 37F20. \\
\end{abstract}

\maketitle

\vspace*{-2\baselineskip}

\section{Introduction}
A classical formula obtained by Steinberg in 1968, c.f. \cite{Steinberg}, shows that the growth series of a Coxeter group (with respect to its standard generating set consisting of involutions) is a rational function, hence the growth rates of these groups are algebraic numbers. In 1980's Cannon discovered a remarkable connection between Salem polynomials and growth functions of surface groups and some cocompact Coxeter groups of ranks $3$ and $4$. Even though these results were published much later in the paper \cite{CaWa}, the initial preprint spawned the studies by other authors, who established that in many cases the growth rates of cocompact and cofinite Coxeter groups are either Salem or Pisot numbers with most notable results obtained in the works \cite{Floyd, Parry}. However, the classes of Salem and Pisot numbers appear to be somewhat narrow, since growth rates of many cocompact and cofinite hyperbolic Coxeter groups do not belong there. As it was shown in \cite[Theorem~4.1]{KePe}, in many cases these growth rates reside in a wider class of Perron numbers, and it was conjectured \cite[p. 1301]{KePe} that this is the case for all Coxeter groups acting cocompactly on hyperbolic spaces as reflection groups. The actual conjecture describes a detailed distribution of the poles of the associated growth series, and it implies that the growth rate is a Perron number. Several results confirming the latter fact have appeared recently in \cite{Kolpakov, KoYu, NoKe, Umemoto, Yu1, Yu2}.

The geodesic growth functions of Coxeter groups have also attracted some attention in the recent works \cite{AC, CK}. However, there are fewer methods available for computing them, e.g. there is no analogue of such a convenient tool as Steinberg's formula. Thus, the number-theoretic properties of geodesic growth rates still remain less understood.

In the present work we show that the spherical and geodesic exponential growth rates of infinite right-angled Coxeter groups (RACGs) and right-angled Artin groups (RAAGs) are Perron numbers, besides the cases when they equal $1$. Namely, in the case of RACGs the following theorems hold.

\begin{letteredthm}[A]\label{thm-A}
Let $G$ be an infinite right-angled Coxeter group with defining graph $\Gamma$. Then the spherical exponential growth rate $\alpha(G)$ of $G$ with respect to its standard set of generators determined by $\Gamma$ is either $1$, or a Perron number.
\end{letteredthm}

\begin{letteredthm}[B]\label{thm-B}
Let $G$ be an infinite right-angled Coxeter group with defining graph $\Gamma$. Then the geodesic exponential growth rate $\beta(G)$ of $G$ with respect to its standard set of generators determined by $\Gamma$ is either $1$, or a Perron number.
\end{letteredthm}

Analogous results hold for RAAGs and their growth rates.

\begin{letteredthm}[C]\label{thm-C}
Let $G$ be a right-angled Artin group with defining graph $\Gamma$. Then the spherical exponential growth rate $\alpha(G)$ of $G$ with respect to its standard symmetric set of generators determined by $\Gamma$ is either $1$, or a Perron number.
\end{letteredthm}

\begin{letteredthm}[D]\label{thm-D}
Let $G$ be a right-angled Artin group with defining graph $\Gamma$. Then the geodesic exponential growth rate $\beta(G)$ of $G$ with respect to its standard symmetric set of generators determined by $\Gamma$ is either $1$, or a Perron number.
\end{letteredthm}

The original conjecture by Kellerhals and Perren has been confirmed in several cases \cite{Kolpakov, KoYu, NoKe, Umemoto} by applying Steinberg's formula \cite{Steinberg} and with extensive use of hyperbolic geometry, notably Andreev's theorem \cite{Andreev-1, Andreev-2}. Recently in \cite{Yu1, Yu2} it was established that the growth rates of all $3$-dimensional hyperbolic Coxeter groups are Perron numbers. In the present paper, we prove that the spherical and geodesic growth rates of RACGs and RAAGs are also Perron numbers, even when there is no cocompact or finite covolume action. Moreover, all hyperbolic right-angled polytopes in dimensions $n = 2$ \cite[Theorem 7.16.2]{Beardon} and $n = 3$ \cite{Pogorelov} are classified, while no such polytopes exist in dimensions $n \geq 5$ \cite{PV}. The only known right-angled hyperbolic polytopes in dimension $n = 4$ are the Coxeter $120$-cell \cite{Kellerhals} and its ``garlands'' obtained by glueing several such polytopes along appropriate facets.

Hence, our methods of proof are not related to the geometry of the group action, and rather use the structure of the group considered as a formal language: namely, we consider the corresponding finite state automaton, following the works by Brink and Howlett \cite{BH} and Loeffler, Meier, and Worthington \cite{LMW}. Also, we would like to mention that Theorem A and Theorem C can be deduced from the results of Sections~10--11 in \cite{GTT}, where a different automaton, essentially due to Hermiller and Meier \cite{HM}, has been considered. Much earlier, similar results for the spherical growth rates of partially commutative monoids were obtained in \cite{LR}.

The properties of the automata used in the present work allow us to show the following fact that describes how many geodesics ``on average'' represent an element of word-length $n$. Let us write $a_n \sim b_n$ for a pair of sequences of positive real numbers indexed by integers if $\lim_{n\to \infty} \frac{a_n}{b_n} = 1$.

\begin{letteredthm}[E]\label{thm-E}
Let $G$ be either an infinite right-angled Coxeter group with defining graph $\Gamma$ whose complement $\overline{\Gamma}$ is \textit{not} a union of a complete graph and an empty graph\footnote{Here and further an empty graph means a graph with some (possibly none) vertices and no edges.}, or a right-angled Artin group with defining graph $\Gamma$ that is not empty. Let $a_n$ be the number of elements in $G$ of word-length $n$ with respect to $\Gamma$, and let $b_n$ be the number of length $n$ geodesics issuing from the origin in the Cayley graph of $G$ with respect to $\Gamma$. Then, $b_n \sim C\,\, \delta^n \,\, a_n$, as $n\rightarrow \infty$, where $\delta = \delta(G) > 1$ is a ratio of two Perron numbers, and $C = C(G) >0$ is a constant. In particular, this implies that $\beta(G) > \alpha(G)$.
\end{letteredthm}

As evidenced by our examples in the sequel, geodesic growth rates may not be Perron numbers outside the class of right-angled Coxeter groups. If we consider the automatic growth rate, c.f. \cite{GS}, which is notably associated with a non-standard generating set, then this quantity is not necessarily a Perron number already in the right-angled case.

We refer the reader to the monograph \cite{BB} for a comprehensive exposition of combinatorics of Coxeter groups and to \cite{LM} for more information on the general dynamical properties of finite state automata.

\section{Preliminaries}\label{section:preliminaries}

In this section we briefly recall all the necessary notions and facts that are used in the sequel.

A \emph{Perron number} is a real algebraic integer bigger than $1$ which is greater in its absolute value than any of its other Galois conjugates. Perron numbers constitute an important class of numbers that appear, in particular, in connection with dynamics, c.f. \cite{LM}.

\smallskip

Let $M$ be a square $n \times n$ ($n \geq 1$) matrix with real entries. Then $M$ is called \emph{positive} if $M_{ij} > 0$, for all $1 \leq i, j \leq n$, and \emph{non-negative} if $M_{ij} \geq 0$, for all $1 \leq i, j \leq n$.

A non-negative matrix $M$ is called \emph{reducible (or decomposable)} if there exists a permutation matrix $P$ such that $PMP^{-1}$ has an upper-triangular block form. Otherwise, $M$ is called \textit{irreducible (or indecomposable)}. It is well-known that if $M$ is the adjacency matrix of a directed graph $D$, then $M$ is irreducible if and only if $D$ is strongly connected (i.e. there is a directed path between any two distinct vertices of $D$).

The \emph{$i$-th period} ($1 \leq i \leq n$) of a non-negative matrix $M$ is the greatest common divisor of all natural numbers $d$ such that $(M^d)_{ii} > 0$. If $M$ is irreducible, then all periods of $M$ coincide and equal \emph{the period of} $M$. A non-negative matrix is called \emph{aperiodic} if it has period $1$.
A non-negative matrix that is irreducible and aperiodic is called \emph{primitive}.

The classical Perron-Frobenius theorem implies that the largest real eigenvalue of a square $n\times n$ ($n \geq 2$) non-negative primitive integral matrix is a Perron number, c.f. \cite[Theorem 4.5.11]{LM}.

\smallskip

In our case, the matrix $M$ represents the transfer matrix of a finite-state automaton $\mathcal{A}$ (or its part), which can be viewed as a directed graph. Let $a_l = |\{$ words of length $l$ accepted by $ \mathcal{A} \}|$. Then the \textit{exponential growth rate} of the regular language $L = L(\mathcal{A})$ accepted by $\mathcal{A}$ is defined as $\gamma(L) = \limsup_{l\to \infty} \sqrt[l]{a_l}$. The spectral radius of $M$ equals exactly $\gamma(L)$ provided that the latter is bigger than $1$, c.f. \cite[Proposition 4.2.1]{LM}.

If $G$ is a group with a generating set $S$, let $S^{-1}$ be the set of inverses of the elements in $S$. The word-length of an element in $g \in G$ is the minimum length of a word over the alphabet $S \cup S^{-1}$ needed to write $g$ as a product. Then we define the \textit{spherical exponential growth rate} of $G$ with respect to $S$ as $\alpha(G, S) = \limsup_{l\to \infty}\sqrt[l]{a_l}$, for $a_l$ being the number of elements in $G$ of word-length $l$.

The \textit{geodesic exponential growth rate} $\beta(G)$ of the group $G$ with respect to a generating set $S$ is defined as $\beta(G, S) = \limsup_{l\to \infty}\sqrt[l]{b_l}$, for $b_l$ being the number of geodesic paths in the Cayley graph of $G$ with respect to $S$ starting at the identity and having length $l$. Here a geodesic path is a path joining two given vertices and having minimal number of edges, hence it is simple (i.e. without backtracking or self-intersections) .

If $\mathrm{ShortLex}$ is the shortlex language for $G$ and $\mathrm{Geo}$ is the geodesic language for $G$, in each case with respect to $S$, then $\alpha(G, S) = \gamma(\mathrm{ShortLex})$ and $\beta(G, S) = \gamma(\mathrm{Geo})$.

The right-angled Coxeter group (or RACG, for short) $G$ defined by a simple graph $\Gamma = (V, E)$ with vertices $V = V\Gamma$ and edges $E = E\Gamma$, is the group with standard presentation
\begin{equation*}
G = \langle  v \in V\Gamma \mid v^2=1,\, \mbox{ for all } v \in V\Gamma,\quad [u, v]=1, \mbox{ if } (u, v) \in E\Gamma \rangle,
\end{equation*}
while the right-angled Artin group (or RAAG) $G$  defined by $\Gamma$ has standard presentation
\begin{equation*}
G = \langle  v \in V\Gamma \mid [u, v]=1, \mbox{ if } (u, v) \in E\Gamma \rangle.
\end{equation*}

It is known that the $\mathrm{ShortLex}$ and $\mathrm{Geo}$ languages are regular for RACGs and RAAGs with their standard symmetric generating sets, c.f. \cite{BH, LMW}. In the sequel, for a RACG or RAAG $G$ we shall write simply $\alpha(G)$, resp. $\beta(G)$, for the spherical, resp. geodesic,  growth rate of $G$ with respect to its standard symmetric generating set.

As the complement $\overline{\Gamma}$ of the defining graph $\Gamma$ splits into connected  components, the cor\-res\-pon\-ding RACG or RAAG splits into a direct product of the respective irreducible RACGs or RAAGs. If $\overline{\Gamma}$ has a connected component with three or more vertices, then the growth rate (spherical or geodesic) of the associated RACG is strictly greater than $1$. An analogous statement holds for a RAAG defined by a graph $\Gamma$ such that $\overline{\Gamma}$ has a connected component with two or more vertices. Thus, apart from easily classifiable exceptions, the growth rates (spherical and geodesic) of RACGs and RAAGs are strictly greater than $1$.

We would like to stress the fact that the geodesic growth rate of a Coxeter group (not a RACG) does not have to be a Perron number (even if it is greater than $1$), as the example of the affine reflection group $\widetilde{A}_2$ shows (its spherical growth rate is, however, equal to $1$). The automaton $\mathcal{A}$ recognising the geodesic language $\mathrm{Geo}(\widetilde{A}_2)$ can be found in the book by Bj\"orner and Brenti \cite{BB} on page~118 (Figure~4.9), and is depicted in Figure~\ref{fig:automatonA2tilde} below for reader's convenience.

\begin{figure}[ht]
\centering
\includegraphics[width=13 cm]{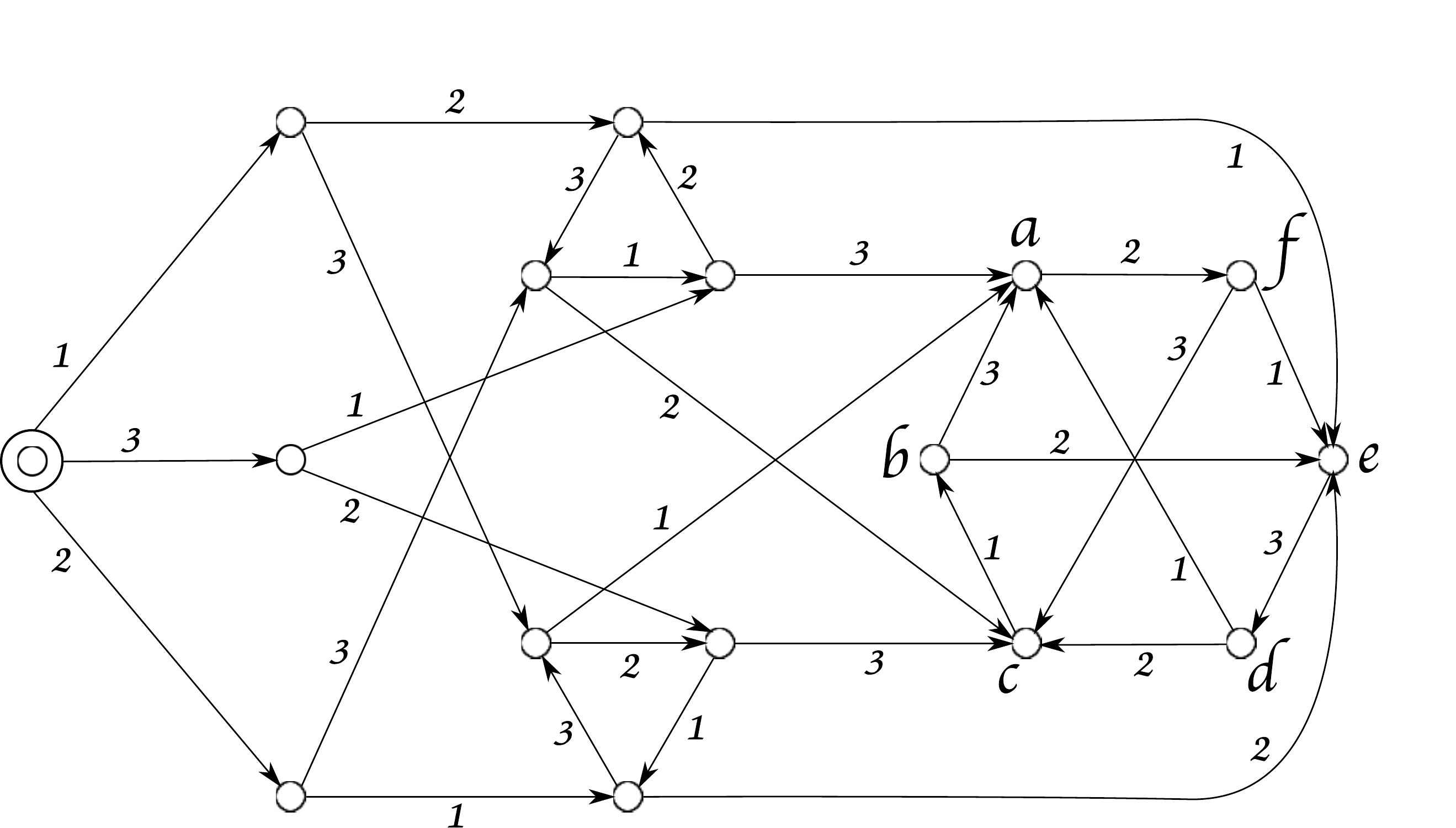}
\caption{\footnotesize The geodesic automaton for $\widetilde{A}_2 = \langle v_1, v_2, v_3 \mid v^2_i=1, i\in \{1,2,3\};\,\,\, (v_i v_j)^3=1, (i,j\in \{1,2,3\},\, i\ne j) \rangle$. The generators labelling its arrows are indicated by their indices. The start state is marked by a double circle. The fail state and the corresponding arrows are omitted. The attracting component has vertices $\{a, b, c, d, e, f\}$.}\label{fig:automatonA2tilde}
\end{figure}

Observe that the automaton $\mathcal{A}$ has a single attracting component spanned by the vertices labelled $\{a, b, c, d, e, f\}$, while the period of $a$ equals $\mathrm{gcd}(4,6) = 2$. By \cite[Exercise 4.5.13]{LM}, this is enough to conclude that the growth rate of $\mathcal{A}$ is not a Perron number. Thus, neither is the geodesic growth rate of $\widetilde{A}_2$. A direct computation shows that it equals $\sqrt{2}$, whose only other Galois conjugate is its negative.

We would like to note that we do not know any example of an infinite Coxeter group such that its spherical growth rate with respect to the standard generating set is not a Perron number, neither equal to~$1$. However, one can find a gainsaying example even for a RACG, when one considers a non-standard generating set. It is generally not known, whether the growth series of Coxeter groups are rational for all generating sets, but for a RACG $G$ a natural generating set with this propery was introduced in the paper \cite{GS}. This generating set, called the \textit{automatic} generating set, consists of all words $b_1 b_2 \ldots b_k$ where $\{b_1,b_2,\ldots,b_k\}$ is a clique in the defining graph of the group. The automaton described in \cite[Remark~5]{GS} accepts the shortlex language of normal forms with respect to the aforementioned alphabet, so that the corresponding growth series is rational and the spherical growth rate is an algebraic number. Let us then consider the group $\mathbb{Z}_2*(\mathbb{Z}_2 \times \mathbb{Z}_2)$ that is defined by the graph on the set of vertices $S=\{a, b, c\}$ having a single edge joining $b$ and $c$. The spherical growth rate with respect to the standard set $S$ can be easily computed and it equals the golden ratio $(1+\sqrt5)/2$, which is a Pisot number (and thus a Perron number). However, the automatic generating set $\{a,b,c,d\}$ from \cite{GS} provides normal forms where the letters $a$ alter with the other three letters $b$, $c$, and $d = bc$, so that the corresponding growth rate equals $\sqrt3$, which is not a Perron number.

\section{Finite automata for shortlex and geodesic words}\label{section:automata}

We begin by introducing more of the general set-up and describing the structure of finite automata for the shortlex and geodesic languages associated with a RACG or RAAG, say $G$. Then we outline the ideas of subsequent proofs. We start by the case of RACGs and then continue to that of RAAGs, since the latter can be deduced from the former.

First of all, certain assumptions can be made about the defining graph $\Gamma = \langle V, E \rangle$ of the RACG $G$ according to our observations in the previous section. We suppose that the complement $\overline{\Gamma}$ is connected and has three or more vertices. Otherwise, either $G \cong D_{\infty}$ or $G$ splits as a direct product of two RACGs $G_1$ and $G_2$, and for the growth rates we have $\alpha(G) = \max\{ \alpha(G_1), \alpha(G_2) \}$ \cite[\S VI.C.59 ]{H} and $\beta(G) = \beta(G_1) + \beta(G_2)$ \cite[Theorem 2.2]{Br}.

If any of $G_i$'s is a finite group, then its defining graph is a complete graph and its spherical and geodesic growth rates are equal to $0$. Otherwise, both of its growth rates are at least $1$. Thus, we either take a maximum of two numbers, each of which is, by assumption, either $0$, or $1$, or a Perron number, or a sum of such two numbers. Thus, the resulting value is also either $0$, or $1$, or a Perron number \cite{Lind}. In fact, $0$ happens as a growth rate for finite RACGs only.

For a RAAG $G$ with defining graph $\Gamma$, we assume that the complement $\overline{\Gamma}$ is connected and has two or more vertices. Otherwise, either $G \cong \mathbb{Z}$, or $G$ splits as a direct product of two RAAGs $G_1$ and $G_2$, and the previous argument for RACGs applies verbatim.  Each $G_i$ has spherical and geodesic growth rates at least $1$.

Now we describe two automata, which are the main objects of our further consideration. The first automaton, called $\mathcal{A}$, accepts the shortlex language of words for the RACG $G$ with respect to its standard generating set, and the second one, called $\mathcal{B}$, accepts the geodesic words for $G$ (with respect to the standard generating set).

We start by describing the automaton $\mathcal{B}$, which is introduced in \cite{LMW}, since it has a simpler structure. For a simple graph $\Gamma$, and a vertex $v\in V\Gamma$, let the star of $v$ be the set $\mathrm{st}(v) = \{ u \in V\Gamma\, |\, u \mbox{ is adjacent to } v  \mbox{ in } \Gamma \}$.

Then, $\mathcal{B}$ has the following set of states $\mathcal{S}$ and transition function $\delta$:
\begin{itemize}
\item[a)] $\mathcal{S} = \{ s \subseteq V\Gamma \mid s \mbox{ spans a clique in } \Gamma  \} \cup \{\emptyset\} \cup \star$,
\item[b)] the start state is $\{\emptyset\}$, and the fail state is $\star$ only, while all other states are accept states,
\item[c)] for each $s \in \mathcal{S}$ and $v \in V\Gamma$ we have $\delta(s, v) = \{v\} \cup (\mathrm{st}(v) \cap s)$, while $v \notin s$, and $\star$ otherwise.
\end{itemize}

Next, we order the vertices of $\Gamma$ with respect to some total order $\{ v_{i_1} < v_{i_2} < \dots < v_{i_n} \}$ and consider the shortlex automaton $\mathcal{A}$ for $G$ which is obtained from $\mathcal{B}$ simply by deleting all the transitions which violate the shortlex order.\footnote{The automaton under consideration is actually accepting the \emph{reverse shortlex} language, where the significance of letters reduces from right to left, with ``smaller'' letters considered more significant. However, this language has the same growth function as the standard shortlex language, and thus there is no difference for the purposes of our proof.}

Thus, we modify $\delta$ as follows:
\begin{itemize}
\item[a)] $\delta(s, v) = \star$, if $v \in s$\, or\, $v > \mathrm{min}(\mathrm{st}(v) \cap s)$, when $\mathrm{st}(v) \cap s \neq \emptyset$,
\item[b)] $\delta(s, v) = \{v\} \cup (\mathrm{st}(v) \cap s)$, otherwise.
\end{itemize}

For the sake of convenience, we shall omit the fail state $\star$ and the corresponding arrows in all our automata, similar to Figure~\ref{fig:automatonA2tilde}.

It is worth noting that the automata $\mathcal A$ and $\mathcal B$ can be built using two different approaches: via the combinatorics on words, where a state describes the set of possible last letters in the normal form of a given word, c.f. \cite{LMW}, or using the geometry of short roots of a given Coxeter group, c.f. \cite{BH} (note that the latter is much more powerful since it works for all Coxeter groups).

In what follows, we shall prove that the transfer matrix $M = M(\mathcal{A} \setminus \{\emptyset\})$ is primitive. We need to consider such a pruned automaton since the start state $\{\emptyset\}$ has no incoming arrows, and thus $\mathcal{A}$ itself is not strongly connected. However, we need only the rest of $\mathcal{A}$  in order to count non-trivial words, and may instead suppose that we have several start states, while the set of accepted words will be partitioned by their first letters.

Then we show that $\mathcal{A} \setminus \{\emptyset\}$ is strongly connected by finding a subset of the so-called \textit{singleton states}, and first showing that the latter is strongly connected (Lemma \ref{singletons-connected}). Then we prove that for any other state there is always a directed path in $\mathcal{A} \setminus \{\emptyset\}$ leading to a singleton state (Lemma \ref{level}) and vice versa (an easy observation). This is equivalent to saying that $M$ is irreducible.

Furthermore, at least one of the singleton states belongs simultaneously to a $2$- and a $3$-cycle of directed edges in $\mathcal{A}$ (Lemma \ref{period-one}). This will imply that $M$ is aperiodic. Then the Perron-Frobenius theorem, as stated in \cite[Theorem 4.5.11]{LM}, applied to $M$ guarantees that $\alpha(G)$ is a Perron number (Theorem A). By applying analogous reasoning to the automaton $\mathcal{B}$, we obtain that $\beta(G)$ is also a Perron number (Theorem B).

In order to proceed to RAAGs, we apply \cite[Lemma 2]{DS} stating that  for a RAAG $G$, there exists an associated RACG $G^\pm$ whose spherical and geodesic growth rates coincide with those of $G$. Thus, the result for RAAGs follows (Theorems C and D).

Finally, by using the notion of matrix domination \cite[Definition A.7]{B}, we show that the Perron-Frobenius eigenvalue of the transfer matrix of $\mathcal{A}$ strictly dominates that of $\mathcal{B}$ under certain simple conditions on the defining graph $\Gamma$, from which the required inequality for the growth rates immediately follows (Theorem E).

\section{Proof of Theorem A}\label{section:proof-A}

Let $G$ be an infinite right-angled Coxeter group with defining graph $\Gamma$. We show that \emph{the spherical exponential growth rate $\alpha(G)$ of $G$ with respect to its standard set of generators determined by $\Gamma$ is either $1$ or a Perron number.}

In the sequel, we assume that $\overline{\Gamma}$ is connected, otherwise we proceed to its connected components, as discussed in the previous section. Also, let $\Gamma$ have at least $3$ vertices, otherwise $G \cong D_\infty$ and the proof is finished.

The following definition describes a useful class of states of the shortlex automaton $\mathcal{A}$ introduced in the previous section.
\begin{defn}
Let $s \in \mathcal{S}$ be a state of the automaton $\mathcal{A}$. We call $s$ a \textit{singleton} if $s = \{v\}$ for a vertex $v \in V\Gamma$.
\end{defn}

Next, we show a crucial, albeit almost evident, property of singleton states.
\begin{lemma}\label{singletons-connected}
The set of singleton states of $\mathcal{A}$ is strongly connected.
\end{lemma}
\begin{proof}
If two vertices $u$ and $v$ are connected in $\overline{\Gamma}$, then $\delta(\{u\}, v) = \{v\}$ and $\delta(\{v\}, u) = \{u\}$. By connectivity of $\overline{\Gamma}$, the claim follows.
\end{proof}

With the above lemma in hand, one can prove that the whole $\mathcal{A}\setminus \{\emptyset\}$ is strongly connected. To this end, let us partition the states of $\mathcal{A}$ by cardinality: $\mathcal{S} = \bigsqcup^{m}_{k = 0} \bigcup_{|s| = k} s$, and say that a state $s$ belongs to level $k$ if $|s| = k$, $(0\leq k \leq m)$, where $m$ is the maximal clique size in $\Gamma$.

\begin{figure}[ht]
\centering
\includegraphics[width=9 cm]{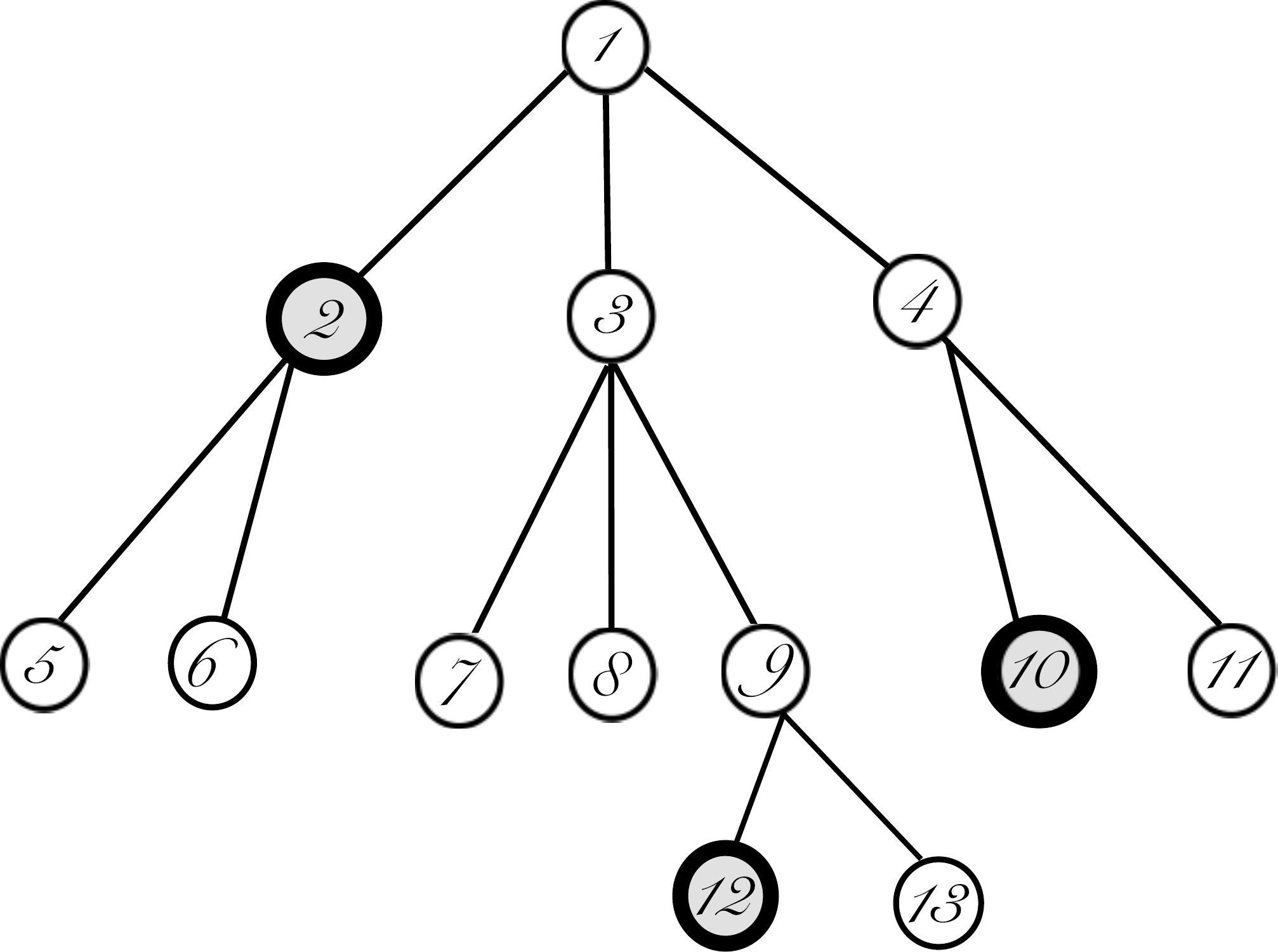}
\caption{\footnotesize A state $s = \{ 2, 10, 12 \}$ of level $3$ is represented by highlighted vertices in the spanning tree $T$ for $\overline{\Gamma}$. Here, following the proof notation, $u = 2$, $v = 10$, and $w = 1$.}\label{fig:tree}
\end{figure}

\begin{lemma}\label{level}
Any state of $\mathcal{A}$ of level $k > 1$ is connected by a directed path to a state of strictly smaller level $l < k$.
\end{lemma}
\begin{proof}
Let us choose a spanning tree $T$ in $\overline{\Gamma}$ and suspend it by the root. We can assume that the order on the vertices of $\Gamma$ is defined by assigning a unique integer label in the set $\{1, \dots, n\}$ and then comparing the numbers in the usual way. We label the root $1$, and the lower levels of successor vertices of $T$ are labelled left-to-right in the increasing order. An example of such labelling is shown in Figure~\ref{fig:tree}.

Let $s \in \mathcal{S}$ be a state of $\mathcal{A}$ (represented by a clique in $\Gamma$) that is not a singleton. Let $u$, $v$ be such vertices in $s$ that $u = \min(s)$,  and $v = \min(s \setminus \{ u \})$. Then there exists a path $p$ in $T$ that connects $u$ and $v$. Necessarily, the length of $p$ is $|p| \geq 2$.

Note that for the vertex $w$ adjacent to $u$ in $p$ we have $w < v$ by construction of $T$ and its labelling. Since $u \notin \mathrm{st}(w)$, we have that $w < \min( \mathrm{st}(w) \cap s )$, and therefore $s' = \delta(s, w) = \{ w \} \cup (\mathrm{st}(w) \cap s) \neq \star$. Thus, we find a new state $s'$ which is not a fail state. If $l = |s'| < |s| = k$, then the proof is finished. Note, that the inequality $l < k$ always holds if $|p| = 2$, since in this case $s' = (s \setminus \{ u,v \})\cup \{ w \}$. Let us suppose that $l = k$ and $|p| > 2$. Then  $s' = (s \setminus \{ u \}) \cup \{ w \}$, and $\min(s') = w$, $\min(s'\setminus \{w\}) = v$, while the path $p'$ joining $w$ to $v$ in $T$ has length $|p'| < |p|$. Hence, we conclude the proof by induction on the length of this path.
\end{proof}

Let $s, s' \in \mathcal{S}$ be two states of $\mathcal{A}$ of the respective levels $l$ and $m$, with $l, m \geq 1$. Then we can apply Lemma~ \ref{level} repeatedly in order to move from $s$ to some singleton state $\{u\}$, while by construction of $\mathcal{A}$ there exist a singleton $\{v\}$ and a directed path from $\{v\}$ to $s'$. Due to  Lemma~\ref{singletons-connected} one can then move among the singletons from $\{u\}$ to $\{v\}$, and thus connect $s$ to $s'$ by a directed path in $\mathcal{A}\setminus \{ \emptyset \}$. Then $\mathcal{A}\setminus \{ \emptyset \}$ is strongly connected, and its transfer matrix $M$ is irreducible.

\begin{lemma}\label{period-one}
At least one singleton state of the automaton $\mathcal{A}$ belongs simultaneously to a $2$- and a $3$-cycle of directed edges in $\mathcal{A}$.
\end{lemma}
\begin{proof}
Since $\overline{\Gamma}$ is connected and has at least $3$ vertices, it contains a path subgraph with vertices $u$, $v$, and $w$, such that $uv$ and $vw$ are edges, and $u > w$ in the lexicographic order. Then we have the following cycles by applying $\delta$:
\begin{itemize}
\item[a)] $ \{u\} \rightarrow \delta(\{u\}, v) = \{v\} \rightarrow \delta(\{v\}, u) = \{u\}$,
\item[b)] $\{u\} \rightarrow \delta(\{u\}, w) = \{ u,w \} \rightarrow \delta(\{ u,w \}, v) = \{v\} \rightarrow \delta(\{v\}, u) = \{u\}$, \\ if $u$ and $w$ commute,
\item[c)] $\{u\} \rightarrow \delta(\{u\}, w) = \{ w \} \rightarrow \delta(\{w \}, v) = \{v\} \rightarrow \delta(\{v\}, u) = \{u\}$, \\ if the element $uw$ has infinite order.
\qedhere
\end{itemize}
\end{proof}

The above statement is equivalent to the transfer matrix $M = M(\mathcal{A}\setminus \{ \emptyset \})$ being aperiodic. Taking into account that $M$ is also irreducible, we obtain that  $M$ is primitive, and its Perron-Frobenius eigenvalue is thus a Perron number by \cite[Theorem 4.5.11]{LM}. In other words, the spherical growth rate $\alpha(G)$ is a Perron number.

\section{Proof of Theorem B}\label{section:proof-B}
Let $G$ be an infinite right-angled Coxeter group with defining graph $\Gamma$. We show that \emph{the geodesic exponential growth rate $\beta(G)$ of $G$ with respect to its standard set of generators determined by $\Gamma$ is either $1$ or a Perron number.}

As indicated in Section~\ref{section:automata}, we may suppose that $\overline{\Gamma}$ is connected and has at least three vertices. Let $\mathcal{B}$ be the geodesic automaton for $G$. Then $\mathcal{B}\setminus \{ \emptyset \}$ is strongly connected, since in order to obtain $\mathcal{B}$ from $\mathcal{A}$ we add directed edges to $\mathcal{A}$, and never remove one. Also, the argument of Lemma~\ref{period-one} applies verbatim to $\mathcal{B}$. Thus, the growth rate of the language accepted by $\mathcal{B}$ is a Perron number.

\section{Proof of Theorems C and D}\label{section:proof-C-D}
Let $G$ be a RAAG with defining graph $\Gamma$ and symmetric generating set $S = \{ v\, :\, v \in V\Gamma \} \cup \{ v^{-1}\, :\, v \in V\Gamma \}$. Let $\alpha(G)$ and $\beta(G)$ be, respectively, the spherical exponential growth rate and the geodesic exponential growth rate of $G$ with respect to $S$. Below we show that \emph{each of $\alpha(G)$ and $\beta(G)$ is either $1$, or a Perron number.}

According to our observation about the behaviour of growth rates of RAAGs with respect to direct products, we may assume that $\overline{\Gamma}$ is connected. By assuming that $\Gamma$ has two or more vertices we guarantee that the spherical and geodesic growth rates of $G$ are strictly greater than $1$. It is well-known, e.g. by \cite[Lemma 2]{DS}, that there exist a RACG $G^\pm$ with generating set $S^\pm$ such that its elements of length $k$ map injectively into the elements of length $k$ in the group $G$ with respect to the generating set $S$.

Indeed, let $G^\pm$ be the associated RACG with defining graph $\Gamma^\pm$, which is the double of $\Gamma$. That is, $\Gamma^\pm$ has a pair of vertices $v^+$ and $v^-$ for each vertex $v$ of $\Gamma$, and if $(u, v)\in E\Gamma$, then $(u^+, v^+)$, $(u^-, v^-)$, $(u^+, v^-)$, $(u^-, v^+)$ are edges of $\Gamma^\pm$. The generating set for $G^\pm$ is $S^\pm = V\Gamma^\pm$. Once we have a word $w = v^{r_1}_{i_1} v^{r_2}_{i_2} \ldots v^{r_s}_{i_s}$ in $G$, consider the corresponding word $\sigma(w) = \prod^{s}_{j=1} \sigma(v^{r_j}_{i_j})$, where each $\sigma(v^{r_j}_{i_j})$ has length $|r_j|$ and alternating form $v^+_{i_j} v^-_{i_j}v^+_{i_j} \ldots v^{\varepsilon}_{i_j}$, if $r_i>0$, or $v^-_{i_j} v^+_{i_j}v^-_{i_j}\ldots v^{-\varepsilon}_{i_j}$, if $r_i<0$, where $\varepsilon = \pm 1$, as appropriate. It is easy to check that the correspondence $\sigma$ between the set of words in $\mathrm{Geo}(G)$ and $\mathrm{Geo}(G^\pm)$ is one-to-one and length-preserving.

Define a lexicographic order on the symmetric generating set $S$ of $G$ in which generators with positive exponents always dominate, i.e. $u > v^{-1}$ for all $u, v \in V\Gamma$, and generators having same sign exponents are compared with respect to some total order such that $u < v$ if and only if $u^{-1} > v^{-1}$, for all $u \neq v \in V\Gamma$. Let the corresponding lexicographic order on the generating set $S^\pm$ of $G^\pm$ be defined by $u^+ > v^-$ for all the corresponding vertices of $\Gamma^\pm$, and $u^+ < v^+$, resp. $u^- > v^-$, whenever $u < v$ in the total order on the generating set $S$. Then $\sigma$ becomes compatible with the corresponding shortlex orders on $G$ and $G^\pm$.

That is, we have a one-to-one correspondence between the set of words of any given length in $\mathrm{Geo}(G)$ and $\mathrm{Geo}(G^\pm)$, as well as in $\mathrm{ShortLex}(G)$ and $\mathrm{ShortLex}(G^\pm)$.  This fact implies that $\alpha(G) = \alpha(G^\pm)$ and $\beta(G) = \beta(G^\pm)$, and thus the spherical growth rate $\alpha(G)$ of $G$ and its geodesic growth rate $\beta(G)$ are Perron numbers, by Theorem A and Theorem B for RACGs.

\section{Proof of Theorem E}\label{section:proof-E}

Let $G$ be an infinite right-angled Coxeter group with defining graph $\Gamma$ such that $\overline{\Gamma}$ is not a union of a complete graph and an empty graph, or let $G$ be a right-angled Artin group, with $\Gamma$ non-empty. Then we show that \textit{the geodesic growth rate $\beta(G)$ strictly dominates the spherical growth rate $\alpha(G)$.} In fact, this statement takes a more quantitative form, as can be seen below.

To this end, let $G$ be a RACG with defining graph $\Gamma$. If $\overline{\Gamma}$ has $k \geq 1$ connected components $\overline{\Gamma}_i$, $i=1,\dots, k$, then $G$ splits as a direct product $G_1 \times \ldots \times G_k$, where $G_i$ is a subgroup of $G$ determined by the subgraph $\Gamma_i$ spanned in $\Gamma$ by the vertices of $\overline{\Gamma_i}$. As mentioned in Section~\ref{section:automata}, the following equalities hold for the spherical and geodesic growth rates of a direct product:
\begin{equation*}
\alpha(G) = \max_{i=1,\dots, k}  \alpha(G_i),
\end{equation*}
while
\begin{equation*}
\beta(G) = \sum^k_{i=1} \beta(G_i),
\end{equation*}
Note that if $\overline{\Gamma_i}$ is an isolated vertex, then $\alpha(G_i) = \beta(G_i) = 0$, otherwise $\alpha(G_i)$, $\beta(G_i) \geq 1$.

\smallskip

Thus, if more than one connected component of $\overline{\Gamma}$ is not a vertex, then $\alpha(G) < \beta(G)$. The equality clearly takes place when $\overline{\Gamma}$ is a union of a complete graph and an empty graph. Now suppose that $\overline{\Gamma}$ is a union of several isolated vertices $v_i$, $i=1, \dots, k$, for $k\geq 0$, and a single connected graph $\overline{\Gamma_0}$ on two or more vertices. Since the non-zero growth rate in this case belongs to the latter, the initial group $G$ can be replaced by its subgroup determined by $\Gamma_0$. Thus we continue by setting $\Gamma := \Gamma_0$, and let $G$ be the corresponding RACG.

Let $M$ be the transfer matrix of the automaton $\mathcal{A}$ (the shortlex automaton for $G$), and $N$ be the transfer matrix of the automaton $\mathcal{B}$ (the geodesic automaton for $G$) constructed in Section~\ref{section:automata}. Since $\mathcal{A}$ is a subgraph of $\mathcal{B}$, if both are considered as labelled directed graphs, then $M$ is dominated by $N$ in the sense of \cite[Definition A.7]{B}. The spherical growth rate $\alpha = \alpha(G)$ and the geodesic growth rate $\beta = \beta(G)$ are the Perron-Frobenius eigenvalues (or, which is the same, spectral radii) of $M$ and $N$, respectively, c.f. \cite[Proposition 4.2.1]{LM}.

As we know from Sections \ref{section:proof-A} and \ref{section:proof-B}, both matrices $M$ and $N$ are irreducible. Moreover, $M$ and $N$ can coincide if and only if there are no commutation relations between the generators of $G$ (i.e. $G$ is a free product of two or more copies of $\mathbb{Z}_2$), which is not the case. Then, by \cite[Corollary A.9]{B}, we obtain the inequality $\alpha < \beta$.

Let $a_n$ be the number of elements in $G$ of word-length $n$ with respect to $\Gamma$, and let $b_n$ be the number of length $n$ geodesics issuing from the origin in the Cayley graph of $G$ with respect to $\Gamma$. Then, since the Perron-Frobenius eigenvalue is simple, the quantities $a_n$ and $b_n$ asymptotically satisfy $a_n \sim C_1\,\,\alpha^n$ and $b_n \sim C_2\,\, \beta^n$, as $n\rightarrow \infty$, for some constants $C_1, C_2 > 0$. Then the claim for RACGs follows.

\smallskip

The case of a RAAG $G$ with defining graph $\Gamma$ such that $\overline{\Gamma}$ is connected can be treated similarly provided the discussion of growth rates in Section~\ref{section:proof-C-D}, and the fact that the corresponding RACG $G^{\pm}$ has empty defining graph if and only if one starts with the empty graph $\Gamma$ for $G$. Otherwise, if $\overline{\Gamma}$ is disconnected, then the geodesic growth rate $\beta = \beta(G)$ is a sum of two or more numbers greater than or equal to $1$ (since the minimal possible spherical or geodesic exponential growth rate equals $1$ for a RAAG), while $\alpha = \alpha(G)$ is the maximum of those, which implies $\alpha < \beta$, as required.

\section*{Acknowledgements}
\noindent
{\small The authors gratefully acknowledge the support that they received from the Swiss National Science Foundation, project no.~PP00P2-170560 (for A.K.), and the Russian Foundation for Basic Research, projects no.~18-01-00822 and no.~18-51-05006 (for A.T.).

A.K. would like to thank Laura Ciobanu (Heriot--Watt University, UK) and A.T. would like to thank Fedor M. Malyshev (Steklov Mathematical Institute of RAS) for stimulating discussions. Both authors feel obliged to Denis~Osin (Vanderbilt University, USA) for his criticism that improved the earlier version of Theorem~E.

Also, A.T. would like to thank the University of Neuch\^{a}tel for hospitality during his visits in August 2018 and May 2019. Both authors express their gratitude to the anonymous referees for their careful reading of the manuscript and numerous comments that invaluably helped to improve the quality of exposition.}

\end{document}